\documentclass[a4paper,11pt]{article}
\usepackage{epsfig}
\usepackage[latin1]{inputenc}
\usepackage{amsmath}
\usepackage{amsfonts}
\usepackage{amssymb}
\usepackage{amsthm}
\usepackage{enumerate}

\newcommand{\C}{\mathbb C}  
\newcommand{\N}{\mathbb N}  
\newcommand{\PP}{\mathbb P}

\newcommand{\G}{\mathbb G}
\newcommand{\A}{\mathbb A}

\newcommand{\vtl}{\, | \,}
\newcommand{\codim}{\mbox{codim}}

\newcommand{\join}{\vee}

\newcommand{\lra}{\longrightarrow}

\newcommand{\dra}{\dashrightarrow}

\newcommand{\de}{\delta}

\newcommand{\si}{\sigma}

\newcommand{\Ga}{\Gamma}
\newcommand{\De}{\Delta}

\newcommand{\La}{\Lambda}

\newcommand{\Si}{\Sigma}
\newcommand{\Tht}{\Theta}

\newcommand{\iso}{\cong}

\newtheorem{coro}{Corollary}
\newtheorem{defin}{Definition}

\newtheorem{lemma}{Lemma}

\newtheorem{prop}{Proposition}

\newtheorem{theo}{Theorem}
\newtheorem*{theo-nonumb}{Multi-secant Lemma}

\begin{document}
\title{Multi-Secant Lemma}
\author{J.Y. Kaminski$^1$, A. Kanel-Belov$^2$ and M. Teicher$^2$\\ \\
$^1$ Holon Institute of Technology, \\
Depart. of Computer Science and Mathematics, \\
Holon, Israel. \\ 
Email: kaminsj@hit.ac.il. \\ \\
$^2$ Bar-Ilan University, \\
Depart. of Mathematics, \\
Ramat-Gan, Israel. \\
Email: {kanel@mccme.ru, teicher@math.biu.ac.il}
}
\date{}
\maketitle


\begin{abstract}
We present a new generalization of the classical trisecant lemma. Our approach is quite different from previous generalizations~\cite{Laudal-77, Ran-91, Adlandsvik-87, Adlandsvik-88, Flenner-all-99, Kaminski-06}. 

Let $X$ be an equidimensional projective variety of dimension $d$. For a given $k \leq d+1$, we are interested in the study of the variety of $k-$secants. The classical trisecant lemma just considers the case where $k=3$ while in \cite{Ran-91} the case $k = d+2$ is considered. Secants of order from 4 to $d+1$ provide service for our main result. 

In this paper, we prove that if the variety of $k-$secants ($k \leq d+1$) satisfies the three following conditions: (i) trough every point in $X$, passes at least one $k-$secant, (ii) the variety of $k-$secant satisfies a strong connectivity property that we defined in the sequel, (iii) every $k-$ secant is also a $(k+1)-$secant, then the variety $X$ can be embedded into $\PP^{d+1}$. The new assumption, introduced here, that we called strong connectivity is essential because a naive generalization that does not incorporate this assumption fails as we show in some example. 

The paper concludes with some conjectures concerning the essence of the strong connectivity assumption.
\end{abstract}

\section{Introduction}

The classic trisecant lemma states that if $X$ is an integral curve of $\PP^3$ then the variety of trisecants has dimension one, unless the curve is planar and has degree at least $3$, in which case the variety of trisecants has dimension 2. Several generalizations of this lemma has been considered \cite{Laudal-77, Ran-91, Adlandsvik-87, Adlandsvik-88, Flenner-all-99, Kaminski-06}. In \cite{Laudal-77}, the case of an integral curve embedded in $\PP^3$ is further investigated leading to a result on the planar sections of the such a curve. On the other hand, in \cite{Ran-91}, the case of higher dimensional varieties, possibly reducible, is inquired. For our concern, the main result of \cite{Ran-91} is that if $m$ is the dimension of the variety, then the union of a family of $(m+2)$-secant lines has dimension at most $m+1$. A further generalization of this result is given in \cite{Adlandsvik-87, Adlandsvik-88, Flenner-all-99}. In this latter case, the setting is the following. Let $X$ in an irreducible projective variety over an algebraically closed field of characteristics zero. For $r \geq 3$, if every $(r-2)-$plane $\overline{x_1 ... x_{r-1}}$, where the $x_i$ are generic points, also meets $X$ in an $r-$th point $x_r$ different from $x_1,...,x_{r-1}$, then $X$ is contained in a linear subspace $L$, with $\codim_L X \leq r-2$.

Here we investigate the case of lines that intersect the variety $X$ (supposed to be equidimensional) in $m$ points such that $m \leq \dim(X)+1$. We prove the following theorem. 
\begin{theo-nonumb}
Consider an equidimensional variety $X$, of dimension $d$. For $m \leq d+1$, if the variety of $m-$secant satisfies the following assumption:
\begin{enumerate}
\item through every point in $X$ passes at least one $m-$secant,
\item the variety of $m-$secant is strongly connected,
\item every $m-$secant is also a $(m+1)-$secant,
\end{enumerate}
then the variety $X$ can be embedded in $\PP^{d+1}$.
\end{theo-nonumb}

Roughly speaking, strong connectivity means that two $m-$secants $l_1$ and $l_2$ can be joined by a finite sequence $\{(p_i,u_i)\}_{i=1,..,n}$ where $u_1=l_1$, $u_n = l_n$, and each line $u_i$ is a $m-$secant passing though $p_i \in X$. A precise statement is given in definition~\ref{def::strong_connectivity}. This condition is not only technical, but really essential because a naive generalization of the trisecant lemma fails as the following example shows. 

{\bf Example 1} {\it Consider the four circles $C_1,C_2,C_3,C_4$ in $\C^3$ respectively defined by} 
$$
\begin{array}{c}
\{z=0,x^2+y^2-1=0\}, \\
\{z=1,x^2+y^2-1=0\}, \\
\{z=2,x^2+y^2-1=0\}, \\
\{z=3,x^2+y^2-1=0\}.
\end{array}
$$ 
{\it Let $Q_0$ be the cylinder defined by $x^2+y^2-1=0$. Consider now the surfaces $S_1,S_2,S_3,S_4$ obtained by the following product $S_i = C_i \times \C$. These surfaces are embedded into $\C^4$. We consider their Zariski closure $Y_1,Y_2,Y_3,Y_4$ in $\PP^4$. Let $Q_1 = Q_0 \times \C$ and $Q$ its Zariski closure in $\PP^4$. Let $S$ be the set of lines contained in $Q$. It can be shown easily that the variety $S$ is not strongly connected but satisfies the other two conditions of the multi-secant lemma (with $m = 3$), while the union of surfaces $S_1 \cup S_2 \cup S_3 \cup S_4$ is not embedded in $\PP^3$.} 

Through this paper, we deal with complex algebraic varieties or equivalently with varieties defined over an algebraically closed field of zero characteristic. However, our considerations and approach are purely algebraic. It is worth noting the results require the field being of zero characteristics. Indeed it is well known that the trisecant lemma is not true in positive characteristics, as shown in an example due to Mumford and published in~\cite{Shokurov}.


The paper is organized as follows. For sake of completeness, in section~\ref{sec::background} we mostly recall standard material and introduce some definitions, we use in the sequel. Then section~\ref{sec::core} is the core of the paper and contains the main results. 

\section{Notations and Background}
\label{sec::background}

In this section, we recall some standard material on incident varieties, that will be used in the sequel.

\subsection{Variety of Incident Lines}

Let $\G(1,n) = G(2,n+1)$ be the Grassmannian of lines included in $\PP^n$. Note that we used $\G$ for the projective entity and $G$ for the affine case. Remind that $\G(1,n)$ can be canonically embedded in  $\PP^{N_1}$, where $N_1=\binom{2}{n+1}-1$, by the Pl\"ucker embedding and that $\dim(\G(1,n))=2n-2$. Hence a line in $\PP^n$ can be regarded as a point in $\PP^{N_1}$, satisfying the so-called Pl\"ucker relations. These relations are quadratic equations that generate a homogeneous ideal, say $I_{\G(1,n)}$, defining $\G(1,n)$ as a closed subvariety of $\PP^{N_1}$. Similarly the Grassmannian, $\G(k,n)$, gives a parametrization of the $k-$dimensional linear subspaces of $\PP^n$. As for $\G(1,n)$, the Grassmannian $\G(k,n)$ can be embedded into the projective space $\PP^{N_k}$, where $N_k=\binom{k+1}{n+1}-1$. Therefore for a $k-$dimensional linear subspace, $K$, of $\PP^n$, we shall write $[K]$ for the corresponding projective point in $\PP^{N_k}$. 

\begin{defin}
Let $X \subset \PP^n$ be an irreducible variety. We define the following variety of incident lines:
$$
\De(X) = \{l \in \G(1,n) \vtl l \cap X \neq \emptyset \}.
$$
\end{defin}

The codimension $c$ of $X$ and the dimension of $\De(X)$ are related by the following lemma.

\begin{lemma}
\label{lemma::delta_dim}
Let $X \subset \PP^n$ be an irreducible closed variety of codimension $c \geq 2$. Then $\De(X)$ is an irreducible variety of $\G(1,n)$ of dimension $2n-1-c$.
\end{lemma}
This lemma is quite a standard. A proof can be found in our paper on trisecant lemma for non-equidimensional varieties \cite{Kaminski-06}. 


The following simple result will be useful in the sequel.

\begin{lemma}
\label{lemma::no_delta_inclusion}
Let $X_1$ and $X_2$ be two irreducible closed varieties in $\PP^n$ of codimension greater or equal to $2$. Then $\De(X_1) \not\subset \De(X_2)$ unless $X_1 \subset X_2$. 
\end{lemma}
A proof of this lemma can also be found in our previous paper \cite{Kaminski-06}.


\subsection{Join Varieties}

Consider $m < n$ closed irreducible varieties $\{Y_i\}_{i=1,...,m}$ embedded in $\PP^n$, with codimensions $c_i \geq 2$. Consider the {\it join variety}, $J = J(Y_1,...,Y_m) = \De(Y_1) \cap ... \cap \De(Y_m)$, included in $\G(1,n)$. We assume that $\sum_{i=1,...,m} c_i \leq 2n-2+m$, so that $J$ is not empty. We shall first determine the irreducible components of $J$. 

Let $U$ be the open set of $Y_1 \times ... \times Y_m$ defined by $\{(p_1,...,p_m) \in Y_1 \times ... \times Y_m \vtl \exists i \neq j, p_i \neq p_j \}$.  Let $V$ be the locally closed set made of the $m-$tuples in $U$, which points are collinear. Let $s: V \lra \G(1,n)$ be the morphism that maps a $m-$tuple of aligned points to the line they generate. Let $S \subset \G(1,n)$ be the closure of the image of $s$. 

First let us look at the irreducible components of $S$. These components could be classified in several classes according to the number of distinct points in the $m-$tuples that generate them. For example consider the case where $m=3$. The locally closed subset of $Y_1 \times Y_2 \times Y_3$, made of triplets of three distinct and collinear points, generates one component of $S$. Now, if $Y_{12}$ is an irreducible component of $Y_1 \cap Y_2$ not contained in $Y_3$, then the lines generated by a point of $Y_{12} \backslash Y_3$ and another point in $Y_3$ form also an irreducible component of $S$. Also let $Z$ be an irreducible component of $Y_1 \cap Y_2 \cap Y_3$, then the lines generates by a point of $Z$ and another point in $Y_1$ are the intersection of the secant variety of $Y_1$ with $\De(Z)$, and form an irreducible component of $S$ too. In the general case, the following lemma will be enough for our purpose. 

\begin{lemma}
The irreducible components of $J$ are:
\begin{enumerate}
\item $\De(Z)$, where $Z$ runs over all irreducible components of $Y_1 \cap ... \cap Y_m$,
\item the irreducible components of $S$, which are not included in any component of the form $\De(Z)$.
\end{enumerate}
\end{lemma}
This lemma was previously introduced in our paper \cite{Kaminski-06}. We refer to it for a proof.

For simplicity, we shall call the irreducible components of $S$ {\it joining components} of $J$ and components of the form $\De(Z)$ for some irreducible component $Z$ of $Y_1 \cap ... \cap Y_m$, {\it intersection components}. 

Before closing this section, we shall clarify an important matter of terminology.  For this purpose and throughout the paper, we use the following notations. If $X$ is a projective subvariety of $\PP^n$, we shall write $T_p(X)$ for the projective embedded tangent space of $X$ at $p$. The Zariski tangent space is denoted $\Tht_p(X)$. Let $CX$ be the affine cone over $X$, then $T_p(X)$ is the projective space of one-dimensional subspaces of $\Tht_q(CX)$, where $q \in \A^{n+1}$ is any point lying over $p$. Hence for a morphism $f$ between two projective varieties $X$ and $Y$, which can be also be viewed as a morphism between $CX$ and $CY$, the differential  $df_p: T_p(X)\backslash \PP(\ker(\phi)) \lra T_{f(p)}(Y)$ is induced by the differential $\phi$ between the Zariski tangent spaces $ \Tht_q(CX)$ and $\Tht_{f(q)}(CY)$. For simplicity, we shall write: $df_p: T_p(X) \lra T_{f(p)}(Y)$, while it is understood that $df_p$ might be defined on a proper subset of $T_p(X)$.

Eventually, we quote a theorem that we shall use several times in the sequel. 
\begin{theo}
\label{theo::ran}
For a projective variety $X \subset \PP^n$ (possibly singular and/or reducible), the variety of $(d+2)-$secants of $X$, where $d=\dim(X)$, always fills up at most a $(d+1)-$fold. 
\end{theo}
A broader version and a proof of this theorem were introduced in~\cite{Ran-91}.

\section{Multi-Secant Lemma}
\label{sec::core}

Before we proceed, we need to prove a few preliminary results. Despite these results are rather known, we include them in the paper for the sake of completeness. The following proposition also illustrates the techniques we use in the paper. It can be viewed as a generalization of a well-known result of Samuel,~\cite{Hartshorne-77} page 312, which deals with smooth curves.

\begin{prop}
\label{prop::pencil}
Let $X$ be an irreducible closed subvariety of $\PP^n$ of dimension $k$. If there exists $L \in \G(k-1,n)$, such that for all  points $p \in U_0$, where $U_0$ is a dense open set of $X$, $L \subset T_p(X)$, then $X$ is a $k-$dimensional linear space, containing $L$.
\end{prop}
A proof can be found in our paper \cite{Kaminski-06}. 

Note that this fact does not hold in positive characteristic as the following example shows. Consider the curve in $\PP^3$, over a field $K$ of characteristic $p$, defined by the ideal 
$$
<\!\!y^p-zt^{p-1},x^p-yt^{p-1}\!\!> \subset K[x,y,z,t],
$$ 
with $t=0$ being the plane at infinity. The tangent space at $(x_0,y_0,z_0,t_0)$ is given by the following system of linear equations: 
$$
\left \{ \begin{array}{l}
t_0^{p-1}z+(p-1)z_0t_0^{p-2}t=0 \\ 
t_0^{p-1}y+(p-1)y_0t_0^{p-2}t=0
\end{array} \right.
$$ 
Every two tangent spaces are parallel and therefore they all contain the same point at infinity. However the the curve is not a line. Note that the point $(0,0,1,0)$ is a singular point of the curve.

The next proposition is used throughout the paper several times. The underlying idea is the following. Let $L$ be a $k-$dimensional linear space. If the tangent space to an irreducible variety at a generic point always spans with $L$ a $(k+1)-$dimensional linear space, then the variety itself must be included into a $(k+1)-$dimensional linear space, containing $L$.

\begin{prop}
\label{prop::planar}
Let $X$ be an irreducible closed subset of $\PP^n$, with $\dim(X) = r$. If there exists $L \in \G(k,n)$, such that for all points $p \in U_0$, where $U_0$ is dense open set of $X$, $\dim(L \cap T_p(X)) \geq r-1 $, then $X$ is included in a $(k+1)-$dimensional linear space, containing $L$.
\end{prop}
We initially introduced this lemma in our previous paper~\cite{Kaminski-06}. However, since it is a major importance for the sequel, we present here a proof for the reader's convience. 
\begin{proof}
If $X \subset L$, then there is nothing to prove. Similarly if $dim(X) = 0$, the result is obvious. Therefore let us assume that $X \not\subset L$ and $dim(X) = r > 0$. Let $\si_L \subset \G(k+1,n)$ be the set of $(k+1)-$dimensional linear spaces that contains $L$. Consider the rational map: $f: X \dra \si_L, p \mapsto p \join L$, where $\join$ is the join operator \cite{Barnabei-all-85}, equivalent to the classical exterior product \footnote{As in \cite{Barnabei-all-85}, the departure from the classical notation is amply justified by the geometric meaning on the operator.}. This mapping is defined over the open set $U$ of regular points in $(X \backslash L) \cap U_0$. Each such point is mapped to the $(k+1)-$dimensional space generated by $p$ and $L$. Since $\dim(T_p(X) \cap L) = r-1$, we have the following inclusion $T_p(X) \subset p \join L = f(p)$, for $p \in U$. 
Let $Y$ be the closure of $f(U)$ in $\si_L$. Thus $Y$ is irreducible. 

Since the ground field is assumed to have characteristic zero, there exists a dense open set $V$ of $X$ such that for any point $p$ in $V$, the differential $df_p$ is surjective,~\cite{Hartshorne-77} page 271. 

This differential is simply: $df_p: T_p(X) \lra T_{f(p)}(Y), a \mapsto a \join L$. Since $T_p(X) \subset p \join L$, $df_p$ is constant over $T_p(X) \backslash L$ and takes the value $p \join L = df_p(p)$. Thus $\dim(Y) = 0$. Since $Y$ is irreducible, $Y$ is a single point corresponding to a $(k+1)-$dimensional linear space, say $K$, containing $L$. Therefore $X \subset K$.
\end{proof}

This proposition does not hold in positive characteristic. Indeed over a field of characteristic $p$, for the curve in $\PP^3$ defined by the following ideal:\\ $<\!\!yt^{p-1}-x^p,zt^{p^2-1}-x^{p^2}\!\!>$, all the tangent lines are parallel and therefore intersect in some point at infinity. But the curve is not a line.

\subsection{The Main Result}
\label{sec::main_result}

Consider now irreducible, distinct closed varieties $Y_1,...,Y_m$,  each of dimension $n-2$, embedded in $\PP^n$. Higher codimension will be considered below. Let $S$ be a join component of $J(Y_1,...,Y_m)$. We assume the following condition (i.e. Condition 1 in the Multi-secant Lemma):
\begin{equation}
\label{full_cond}\tag{$\natural$} 
\mbox{{\it For all $i$ and for all $p \in Y_i$, there exists a line $l \in S$, such that $p \in l$.}}
\end{equation}  

We shall prove that if there exists an additional irreducible variety $Y$, of dimension $n-2$, such that $S \subset \De(Y)$, then there is an hyperplane that contains the varieties $Y_1,...,Y_m,Y$.  We proceed in several steps. First observe that $\dim(S) \geq 2n-2-m$. As a matter of fact, in the sequel the notation $\si_p$ is used for the set of lines passing through $p$, and $X_p = S \cap \si_p$.

\begin{lemma}
\label{lemma::fiber}
Consider the variety $W = \bigcup_{l \in S} l$. Then $W$ is an irreducible variety, that strictly contains $Y_i$ for all $i$. Hence it has dimension either $n-1$ or $n$. If $\dim(W) = n-1$, the following facts hold:
\begin{enumerate}
\item For a generic point $p \in W$, we have $\mu = \dim(\si_p \cap S) = 
\dim(S) - n +2 \geq n-m$.

\item For $i=1,...,m$, let $\mu_i = \min_{p \in Y_i} \dim(\si_p \cap S)$ be the dimension of $\si_p \cap S$ for a generic point $p$ of $Y_i$. Then $\mu_1 = \mu_2 = ... = \mu_m \geq n-m$.

\item The variety $W_1 =  \{p \in \PP^n \vtl \dim(\si_p \cap S) \geq \mu+1\}$ has dimension at most $n-3$.
\end{enumerate}
\end{lemma}
\begin{proof}
\begin{enumerate}
\item  Let us consider the following incidence variety:
$$
\Si = \{(l,p) \in S \times W \vtl p \in l\}
$$
endowed with the two canonical projections $\pi_1:\Si \lra S$ and $\pi_2: \Si \lra W$.

For all $l \in S$, the fiber $\pi_1^{-1}(l)$ is irreducible of dimension $1$. Therefore $\Si$ is irreducible and $\dim(\Si) = \dim(S)+1$. Therefore $W$ is also irreducible. Then the set $W$ is an irreducible closed subset of $\PP^n$ which strictly contains each $Y_i$ (otherwise $Y_i = W$). Therefore $\dim(W) \geq n-1$. Note that $W = \pi_2(\pi_1^{-1}(S))$. We have $\mu = \min_{p \in W} \dim(\pi_2^{-1}(p))$.  If $\dim(W) = n-1$, then $\mu = \dim(S)+1-n+1 = \dim(S) - n+2 \geq 2n-2-m - n +2 = n-m$.

\item  If we consider the incidence variety $\Si_i$, defined similarly than $\Si$, except that $W$ is replaced by $Y_i$, then the general fiber of $\pi_1$ is finite (otherwise $Y_i = W$). Thus $\dim(\Si_i) = \dim(S)$ and condition~(\ref{full_cond}) implies that the general fiber of $\pi_2$ has dimension $\dim(S) - n+2 \geq 2n-2-m-n+2 = n-m$.

\item Then $W_1$ is a proper closed subset of $W$. Thus $\dim(W_1) \leq n-2$. If $\dim(W_1) = n-2$, then $\dim(\pi_2^{-1}(W_1)) = n-2 + \mu+1 \geq n-2+n-m+1 = 2n-1+m$, implying that $\pi_2^{-1}(W_1) = \Si$ and so that $W_1 = W$. As a consequence $\dim(W_1) \leq n-3$.
\end{enumerate}
\end{proof}

In addition to condition~(\ref{full_cond}), we need to assume the following {\it strong connectivity}. 

\begin{defin}
\label{def::strong_connectivity}
We shall say that $S$ is {\it strongly connected} if for two lines $l_1,l_2 \in S$, there exists a finite sequence $((p_1,u_1), \hdots, (p_n,u_n))$ that satisfies the following four conditions: 
\begin{enumerate}
\item $\forall i, p_i \in Y_1 \cup \hdots \cup Y_m, u_i \in S$,
\item $u_1 =l_1, u_n = l_2$,
\item $u_i \in X_{p_i} \cap X_{p_{i+1}}$ for $i=1,...,n-1$,
\item $u_i$ and $u_{i+1}$ belongs to same irreducible component of $X_{p_{i+1}}$ for $i=1,...,n-1$.
\end{enumerate}
\end{defin}
This condition implies immediately the following lemma. {\it We shall denote $X_p(\La)$ the union of irreducible components of $X_p$ that intersect a subset $\La$ of $S$.} 

\begin{lemma}
\label{lemma::strong_connectivity}
Consider a line $l_0 \in S$ and $\Lambda_0 = \{l_0\}$. Let $\Gamma_1$ be the variety $\Gamma_1 = l_0 \cap (\cup_{i=1}^m Y_i)$. For each $p \in \Gamma_1$, consider $X_p(\La_0)$, the union of irreducible components of $X_p = S \cap \si_p$ that contain $l_0$. Let $\La_1 = \cup_{p \in \Ga_1} X_p(\La_0)$. Let $\Ga_2 = (\cup_{l \in \La_1} l) \cap (\cup_{i=1}^m Y_i)$. Then define $\La_2 = \cup_{p \in \Ga_2} X_p(\La_1) $. More generally, assume $\La_k$ is defined. Let $\Ga_{k+1}$ be $(\cup_{l \in \La_k} l) \cap (\cup_{i=1}^m Y_i)$ and $\La_{k+1} = \cup_{p \in \Ga_k} X_p(\La_k) $. Then if $S$ is strongly connected, there exists $k_0$ such that $\La_{k_0} = S$.
\end{lemma}
\begin{proof}
 If $\cup_{n \in \N} \La_n \varsubsetneq S$, then every line in $S \setminus \cup_{n \in \N} \La_n$ cannot be reach from $l_0$ as required by the strong connectivity assumption. Thus we necessarily have: $\cup_{n \in \N} \La_n = S$. Since the sequence $\{\La_k\}$ is an increasing sequence of closed subsets in $S$, there must exist $k_0$ such that $\La_{k_0} = S$.
\end{proof}


Now we are in a position to prove the following theorem, which is the basis of the main result that will proved below.

\begin{theo}
\label{theo::main-(n-1)}
Let $m=n-1$ distinct closed varieties $Y_1,...,Y_m$,  each of dimension $n-2$, embedded in $\PP^n$. For a join component $S$ of $J(Y_1,...,Y_m)$ satisfying condition~(\ref{full_cond}) and which is strong connected, if there exists an additional irreducible variety $Y$, distinct from $Y_1,...,Y_m$, of dimension $n-2$, such that $S \subset \La(Y)$, then there is an hyperplane containing $Y_1,...,Y_m,Y$.
\end{theo}
{\it Remark:} Condition~(\ref{full_cond}) and strong connectivity are essential, as shown in Example 1 on page 2.
\begin{proof}
\begin{enumerate}
\item[(i)] Let us consider the variety $X = Y_1 \cup Y_2 \cup ... \cup Y_m \cup Y$. The dimension of $X$ is $n-2$ and every line in $S$  is a $n-$secant of $X$. Then by theorem~\ref{theo::ran}, the union of lines in $S$, denoted $W$ in the lemma~\ref{lemma::fiber}, has dimension at most $n-1$. Then, by lemma~\ref{lemma::fiber}, it has dimension exactly $n-1$.  

Moreover if we consider a generic line $l_0$ of $S$, we can assume that the intersection $Y_i \cap l_0$ is made of smooth points of $Y_i$, for all $i=0,...,m$. Pick $p_i$ in $Y_i \cap l_0$. The tangent spaces $T_{p_i}(Y_i)$ are all $(n-2)-$dimensional. By theorem~\ref{theo::ran}, the set of lines that intersect all the spaces spans a linear space of $\PP^n$ of dimension at most $n-1$. However a short calculation shows that for all $i$ and for all $p \in T_{p_i}(Y_i)$ , there exists such a line that passes through $p$. Then this linear space is actually an hyperplane. We shall denote this hyperplane $H(l_0)$. For all $p_i \in Y_i \cap l_0$, $T_{p_i}(Y_i) \subset H(l_0)$.    

\item[(ii)] By the genericity of the line $l_0$ considered above, we can now assume that there exists a dense open set $U \subset S$, such that for every line $l \in U$, there exists an hyperplane $H(l)$, such that for all $i$, $T_{p_i}(Y_i) \subset H(l)$, where $p_i = Y_i \cap l$.

For a given line $l_0 \in U$, and $p_i = l_0 \cap Y_i$, we have $\dim(X_{p_i}) \geq 1$, by lemma~\ref{lemma::fiber}. Let $X_{p_i}^1, X_{p_i}^2, ..., X_{p_i}^\si$ be the irreducible components of $X_{p_i}$ which contain $l_0$. Since $S$ is strongly connected, there must exist $i$ and $k$ such that $\dim(X_{p_i}^k) \geq 1$. Thus we shall assume $i$ has been chosen so that this condition holds. Let $X_{p_i}(\{l_0\}) = X_{p_i}^1 \cup X_{p_i}^2 \cup ... \cup X_{p_i}^\si$ be the union of these components. Then for each $k$, $U_{p_i}^k = X_{p_i}^k \cap U$ is dense in $X_{p_i}^k$. Let $Z_{p_i}^k = \bigcup_{l \in X_{p_i}^k} l$ be the union of lines in $X_{p_i}^k$, and $\dot{Z}_{p_i}^k = \bigcup_{l \in U_{p_i}^k} l$ be the union of lines in $U_{p_i}^k$. Then $\dot{Z}_{p_i}^k$ is open and dense in $Z_{p_i}^k$. For $j \neq i$, let $D_{p_i,j}^k = Z_{p_i}^k \cap Y_j$ and $\dot{D}_{p_i,j}^k = \dot{Z}_{p_i}^k \cap Y_j$. Thus $\dot{D}_{p_i,j}^k$ is open in $D_{p_i,j}^k$. Let $\bar{D}_{p_i,j}^k$ be the closure of $\dot{D}_{p_i,j}^k$ in $D_{p_i,j}^k$.

For every $j \neq i$ and every $k$ and for all $q \in \dot{D}_{p_i,j}^k$, $T_q(Y_j)$ and $T_{p_i}(Y_i)$ lie in the same hyperplane. This is due to the fact that both $p_i$ and $q \in \dot{D}_{p_i,j}^k$ lie in some line belonging to $U$. 

In particular, $T_q(\bar{D}_{p_i,j}^k)$ intersects $T_{p_i}(Y_i)$ along a linear space of dimension $\dim(T_q(\bar{D}_{p_i,j}^k))+n-2-n+1 = \dim(T_q(\bar{D}_{p_i,j}^k))-1$. By proposition~\ref{prop::planar}, there exits some hyperplane $H_{p_i,j}^k$ containing $T_{p_i}(Y_i)$ and $\bar{D}_{p_i,j}^k$. Thus $\dot{Z}_{p_i}^k \subset H_{p_i,j}$. Then we also have $Z_{p_i}^k \subset H_{p_i,j}^k$. For all $k$, $H_{p_i,j}^k$ is spanned by $T_{p_i}(Y_i)$ and $l_0$. Thus all these hyperplanes coincide. Therefore we shall denote $H_{p_i}$ this hyperplane and $D_{p_i,j} = D_{p_i,j}^1 \cup ... \cup D_{p_i,j}^\si \subset H_{p_i}$ for all $j \neq i$. Therefore it is clear that $H_{p_i} = H(l)$, for any line $l \in Z_{p_i} = Z_{p_i}^1 \cup ... \cup Z_{p_i}^\si$, since $T_{p_i}(Y_i) \subset H(l) \cap H_{p_i}$ and $l \subset H(l) \cap H_{p_i}$. As a consequence $H(l)$ is constant for all $l \in X_{p_i}(\{l_0\})$ and equals $H_{p_i}$.


In order to use lemma~\ref{lemma::strong_connectivity}, we shall consider a particular line $\La_0 = \{l_0\} \subset U$ and $\Ga_1 = l_0 \cap (\cup_{i=1}^m Y_i)$. Then we can conclude that for every $p \in \Ga_1$, $\bigcup_{l \in X_p(\La_0)} l \subset H(l_0)$.  Hence every line in $\La_1 = \cup_{p \in \Ga_1} X_p(\La_0)$ is contained within the plane $H(l_0)$.  

\item[(iii)] Let $l$ be any line in $\La_1$. By the previous argument, $H(l) = H(l_0)$. Then for every $p \in l \cap (\cup_{i=1}^m Y_i)$, the union of lines in $X_p(\{l\})$ is contained within $H(l)=H(l_0)$. Let $\Ga_2 = \Lambda_1 \cap (\cup_{i=1}^m Y_i)$ and  $\La_2 = \cup_{p \in \Ga_2} X_p$. Then we can conclude that every line in $\La_2$ is included in $H(l_0)$. 

By induction, we construct  
$$
\Ga_{k+1} =  \Lambda_k \cap (\cup_{i=1}^m Y_i) \mbox{ and } \La_{k+1} = \cup_{p \in \Ga_k} X_p(\La_k). 
$$
Every line in $\La_k$ lies within the plane $H(l_0)$. By lemma~\ref{lemma::strong_connectivity}, the increasing  sequence $\{\La_k\}$ stabilizes at some stage $k_0$ and $\La_{k_0} = S$.  Then $W = \bigcup_{l \in S} l \subset H(l_0)$. In particular for all $i$, $Y_i \subset H(l_0)$. 
\end{enumerate}
\end{proof}

Let us now consider the case $m < n-1$.

\begin{coro}
\label{coro::m<n}
For $m<n-1$, let $Y_1,...Y_m$ be distinct irreducible subvarieties of $\PP^n$, each of dimension $n-2$. We assume that there exists an additional variety $Y$ also of dimension $n-2$, such that a joining  component $S$ of $J(Y_1,...,Y_m)$, satisfying condition~(\ref{full_cond}) and strongly connected, is included into $\La(Y)$. Then the varieties $Y_1,...,Y_m,Y$ all lie in the same hyperplane.
\end{coro}
\begin{proof}
This corollary follows immediately. It is enough to consider additional varieties $Y_{m+1},...,Y_{n-1}$ such that $S' = S \cap \La(Y_{m+1}) ... \cap \La(Y_{n-1})$ has dimension $n-1$, satisfies condition~(\ref{full_cond}) and is strongly connected. Since $S' \subset \La(Y)$, the varieties $Y_1,....,Y_{n-1}$ must lie in the same hyperplane by theorem~\ref{theo::main-(n-1)}. 
\end{proof}


Now, we shall generalize the theorem~\ref{theo::main-(n-1)} to the case of higher codimension varieties. 

\begin{coro}
\label{coro::higher_codim_hyperplane}
For $m \leq n$, let $Y_1,..., Y_m$ be distinct irreducible subvarieties of $\PP^n$, each of dimension $d \leq n-2$. We assume that $m \leq d+1$ and that there exists an additional variety $Y$ also of dimension $d$, such that a joining component $S$ of $J(Y_1,...,Y_m)$, satisfying condition~(\ref{full_cond}) and strongly connected, is included into $\La(Y)$. Then the varieties $Y_1,...,Y_m,Y$ all lie in the same hyperplane.
\end{coro}
\begin{proof}
We shall proceed by induction on $\de = n-2-d$. For $\de = 0$, it is the content of the previous results. Let us assume that it is true for some $\de$. Then we consider a generic point $p \in \PP^n$, not lying on $X = Y_1 \cup ... \cup Y_m \cup Y$ and so a generic hyperplane $H$ (so not passing through $p$). We project $Y_1,...,Y_m,Y$ onto $H$ trough the center of projection defined by $p$. We obtain $Z_1,...,Z_m,Z$. The variety $S$ yields by this projection a variety $S'$, which is a joining component of $J(Z_1,...,Z_m)$, satisfying condition~(\ref{full_cond}), strongly connected and included in $\La(Z)$. Thus by induction, there exists a $(n-2)-$dimensional  linear space $L$, included in $H$, that contains $Z_1,....,Z_m,Z$. Then the hyperplane generated by $L$ and $p$ contains $Y_1,...,Y_m,Y$. 
\end{proof}

\begin{coro}
\label{coro::higher_codim_optimal}
Let $Y_1,...Y_m$ be distinct irreducible subvarieties of $\PP^n$, each of dimension $d \leq n-2$. We assume that $m \leq d+1$ and there exists an additional variety $Y$ also of dimension $d$, such that a joining component $S$ of $J(Y_1,...,Y_m)$, satisfying condition~(\ref{full_cond}) and strongly connected, is included into $\La(Y)$. Then the varieties $Y_1,...,Y_m,Y$ all lie in the same $(d+1)-$dimensional space.
\end{coro}
\begin{proof}
By corollary~\ref{coro::higher_codim_hyperplane}, we know that $Y_1,...,Y_m,Y$ are all contained in some hyperplane $H \iso \PP^{n-1}$. If there common dimension $d$ is strictly smaller than $n-2$, the procees can be iterated, until we find they are all contained in some $(d+1)-$dimensional linear space.  
\end{proof}

We are now in a postion to formulate the multi-secant lemma.

\begin{theo} {\bf Multi-Secant Lemma}
\label{theo::multi-secant-lemma}
Consider an equidimensional variety $X$, of dimension $d$. For $m \leq d+1$, if the variety of $m-$secant satisfies the following assumption:
\begin{enumerate}
\item through every point in $X$ passes at least one $m-$secant,
\item the variety of $m-$secant is strongly connected,
\item every $m-$secant is also a $(m+1)-$secant,
\end{enumerate}
then the variety $X$ can be embedded in $\PP^{d+1}$. The result still holds even if not the whole variety of $m-$secants satisfies the assumptions, but only a joining component of it.
\end{theo}

\section{Discussion and Conjectures}

One interesting question is to know to which extend the strong connectivity is necessary for the multi-secant lemma to hold. We conjecture the following two propositions:\\
{\it Conjecture 1:} There exist a sequence of varieties $Y_1,...,Y_m$ such that a join component $S$ satisfies conditions 1-3 of the strong connectivity definition, but not the last condition. \\
{\it Conjecture 2:} The fourth condition of the strong connectivity is necessary for the multi-secant lemma.

The second conjecture can be intuitively apprehended by the following considerations. We shall use the same notations that in section~\ref{sec::main_result}. Let $a$ be a smooth point in $Y_1$. There is a curve in $S$, made of lines passing through $a$. The trace of this curve on $Y_2$ is also a curve denoted $C$. Let $b$ be a point on $C$. Through $b$, we can consider a further curve of lines of $S$. This will draws a curve on $Y_1$ through $a$. Now imagine $b$ lies on another connected component of $C$, the same process define another curve on $Y_1$ through $a$. In general these two curves have non parallel tangent vectors at $a$, so that these vectors define a plane. This construction shows that when the fourth condition of the strong connectivity is dropped, two dimensional moves can be constructed. To sum up, these simple considerations show that not assuming the fourth condition of the strong connectivity, will allow considerably more ways to cover the varieties by elementary moves, but will cancel the rigidity needed for enforcing these steps to stay in the same hyperplanes.

\section*{Acknowledgment}

We would like to express our gratitude to Marat Gizatullin for helpful discussions, as well as for giving us the reference of a counter example in positive characteristics~\cite{Shokurov}. 

\bibliographystyle{amsplain}

\end{document}